\documentclass{amsart}
\usepackage{amscd}
\usepackage{amsfonts}
\usepackage{amsmath}
\usepackage{amssymb}
\usepackage{latexsym}
\usepackage{amsthm,color}
\usepackage[all]{xy}
\usepackage{epsfig}
\usepackage{graphicx}
\usepackage{color}
\usepackage{tikz}
\usepackage{bm}
\usetikzlibrary{calc}
\newtheorem{proposition}{Proposition}

\newtheorem{theorem}{Theorem}

\theoremstyle{definition}
\newtheorem{definition}{Definition}
\theoremstyle{definition}
\newtheorem{example}{Example}
\theoremstyle{definition}

\newcommand{\Z}{\mathbb{Z}}
\newcommand{\R}{\mathbb{R}}

\begin{document}
\title[
On envelopes of circle families in the plane]
{On envelopes of circle families in the plane}
\author[Y.~Wang]{Yongqiao Wang
}
\address{
School of Science, Dalian Maritime University, Dalian 116026, P.R. China
}
\email{wangyq@dlmu.edu.cn}
\author[T.~Nishimura]{Takashi Nishimura
}
\address{
Research Institute of Environment and Information Sciences,
Yokohama National University,
Yokohama 240-8501, Japan}
\email{nishimura-takashi-yx@ynu.ac.jp}
\begin{abstract}
In this paper we {\color{black}investigate} the relationships between envelopes
of circle families and some special curves in the plane,
such as evolutes, pedals, evolutoids and pedaloids.
\end{abstract}


\date{}

\maketitle

\section{Introduction\label{section1}}
Envelopes of plane curve families have been well investigated since the beginning
of the history of differential geometry {\color{black}(see for instance \cite{history})},
and they often arise in many guises
in physical sciences. For the most typical case, the straight line families
in the plane have been studied.
{\color{black}In \cite{NIS1} (see also \cite{NIS2} which is
an easy to understand expository article focused on envelopes of line families
in the plane), solving four basic problems (existence problem,
representation problem, uniqueness problem and equivalence problem of definitions),
t}he second author constructed a general theory for envelopes created by straight line
families. The study on circle families in the plane is also important
since the envelopes of them have {\color{black}several} practical applications.
For instance, {\color{black}there is an} application to soil mechanics.
In analysis of the stability of soil masses, the
{\color{black}\it shear strength} $\tau_{f}$ of a soil at a point on a particular plane is
expressed as a linear function of the effective
{\color{black}\it normal stress} $\sigma_{f}$ at failure:
\begin{align*}
\tau_{f}=\sigma_{f}\tan\varphi+c,
\end{align*}
where $\varphi$ and $c$ are the {\it angle of shearing resistance}
and {\it cohesion intercept} respectively.
A method using {\it Mohr circles} to obtain the shear strength parameters $\varphi$
and $c$ {\color{black}can be found}
in \cite{CR}.
{\color{black}According to \cite{CR}, a}
brief description of this method is given as follows.
The stress state of a soil can be represented by a {\color{black}\it Mohr circle}
which is defined by the {\color{black}\it effective principal stresses} $\sigma_1$ and $\sigma_2$. The center and the radii of the Mohr circle are
$(\frac{\sigma_1+\sigma_2}{2},0)$ and $\frac{\sigma_1-\sigma_2}{2}$, respectively.
By experiments, we obtain some values of effective principal stresses $\sigma_1$ and $\sigma_2$ at failure. The Mohr circles in terms of effective principal stress are
drawn in Figure 1.
\begin{figure}[h]
	\begin{center}
		\includegraphics[width=10cm]
		{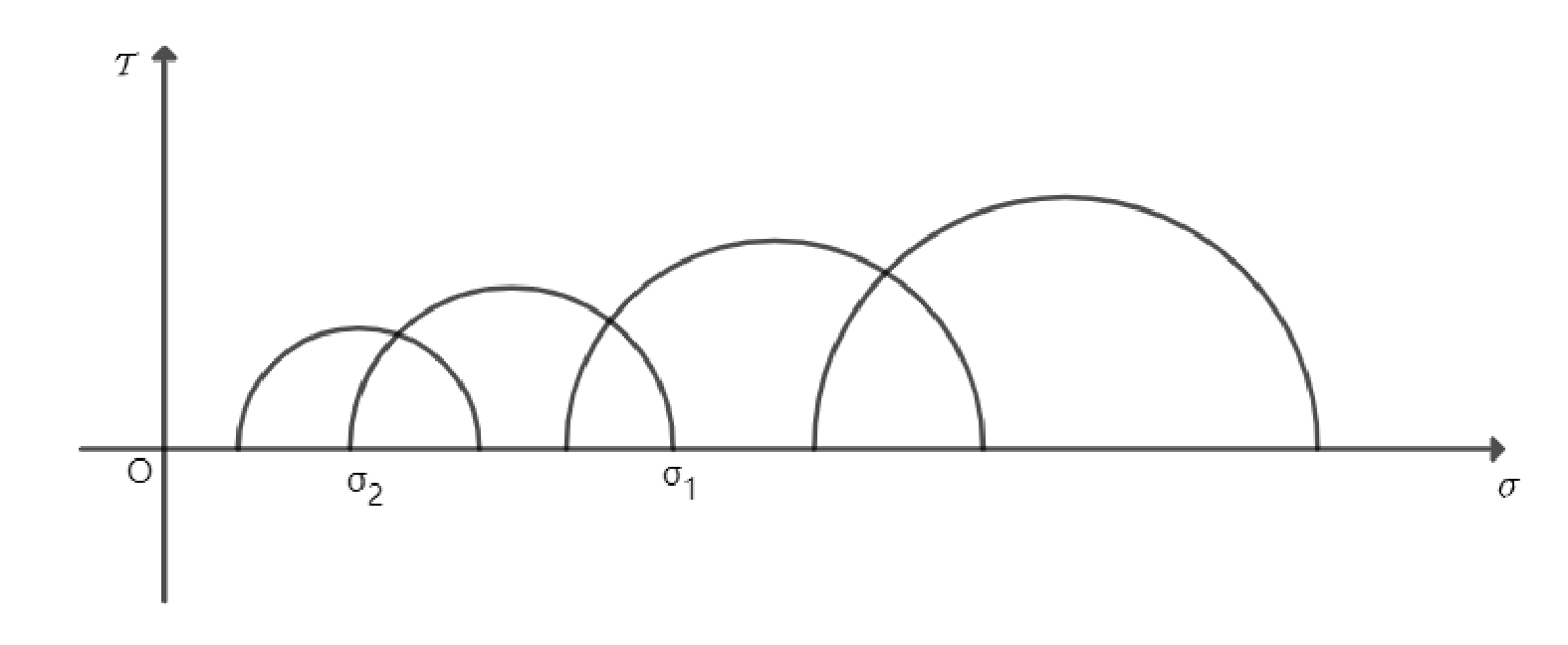}
		\caption{Mohr circles
		}
	\end{center}
\end{figure}
The envelope created by Mohr circles is called the {\color{black}\it Mohr}
{\it failure envelope} which may be a slightly curved curve. Then the shear strength
parameters $\varphi$ and $c$ can be obtained by approximating the curved envelope to a straight line, namely the slope of the straight line equals $\tan\varphi$ and the intercept of straight line on the vertical axis is $c$ (see Figure 2).
\begin{figure}[h]
	\begin{center}
		\includegraphics[width=12cm]
		{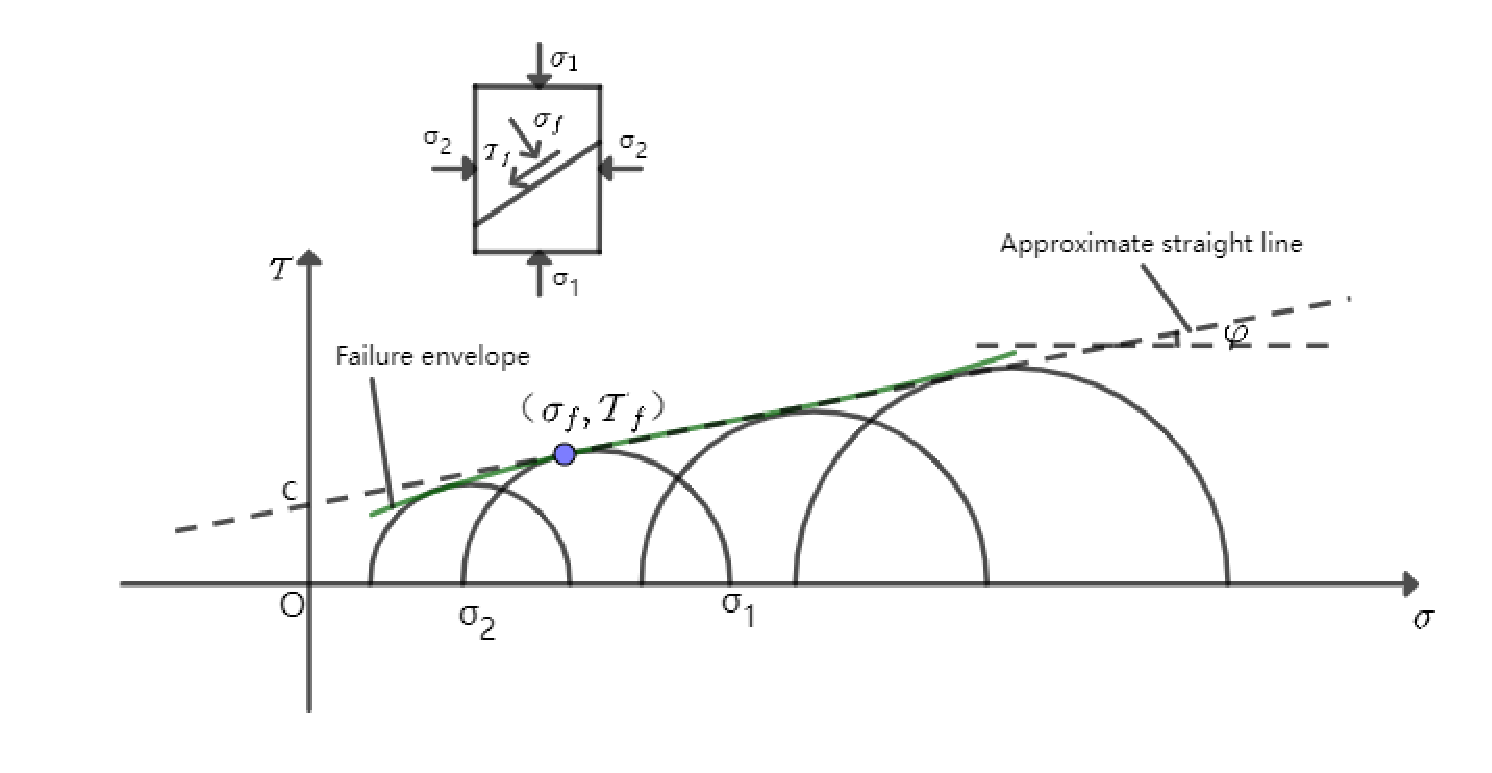}
		\caption{The failure envelope and the approximated straight line
		}
	\end{center}
\end{figure}
Therefore, in order to
investigate the shear strength parameters as precisely as possible, it is important to systematically study the envelopes created by circle families. Moreover, the study of envelopes of circle families can be applied to seismic survey
({\color{black}for example} see \cite{P}).
For these reasons, the authors have constructed a general theory on envelopes created by circle families in the plane {\color{black}in} \cite{WN}.

For a given point $P\in\mathbb{R}^2$ and a positive number
$\lambda\in\mathbb{R}_{+}{\color{black}=\{\lambda\in \R\; |\; \lambda>0\}}$,
the circle centered at $P$ with radius $\lambda$ is naturally defined by
\begin{align*}
	C_{(P,\lambda)}=\{(x,y)\in\mathbb{R}^2|((x,y)-P)\cdot((x,y)-P)=\lambda^2\},
\end{align*}
where the dot {\color{black}in the center}
stands for the standard scalar {\color{black}product of two vectors}.
{\color{black}For an open interval $I$, let}
$\gamma:I\rightarrow\mathbb{R}^2$
{\color{black}(resp., $\lambda:I\rightarrow\mathbb{R}_{+}$ )}
be a {\color{black}$C^\infty$ mapping
(resp., be a $C^\infty$ function)}.   Then,
the circle family $\mathcal{C}_{(\gamma,\lambda)}$ is naturally defined
{\color{black}as follows}.
\begin{align*}
	\mathcal{C}_{(\gamma,\lambda)}=\{C_{(\gamma(t),\lambda(t))}\}_{t\in I}.
\end{align*}
It is reasonable to assume that the normal vector at any point of the curve $\gamma$
is well-defined. Thus,
{\color{black}we naturally reach the following definition.}

\begin{definition}\label{definition1}
	A {\color{black}$C^\infty$
mapping} $\gamma:I\rightarrow\mathbb{R}^2$ is called a
{\color{black}\it frontal} if there exists a
{\color{black}$C^\infty$}
mapping $\nu:I\rightarrow S^1$
such that $(d\gamma/dt)(t)\cdot\nu(t)=0$ for each $t\in I$.
\end{definition}
For a frontal $\gamma$, the mapping $\nu:I\rightarrow S^1$ is called the {\it Gauss mapping} of $\gamma$. We set $\mu(t)=J(\nu(t))$, where $J$ is the anti-clockwise rotation by $\pi/2$. Then we have a moving frame $\{\nu(t),\mu(t)\}$ along the frontal $\gamma(t)$.
Denote $\dot{\nu}(t)\cdot\mu(t)=l(t)$,
{\color{black}where $\dot{\nu}(t)\cdot\mu(t)=
\left(d\nu/dt\right)(t)\in T_{\nu(t)}S^1\subset T_{\nu(t)}\R^2$ and two
vector spaces $\R^2$ and $T_{\nu(t)}\R^2$ are canonically identified.  Then,}
the Frenet formula of $\gamma(t)$ according to the moving frame $\{\nu(t),\mu(t)\}$ is given by
\begin{center}
	$\left(\begin{array}{c}
		\dot{\nu}(t) \\
		\dot{\mu}(t)
	\end{array}
	\right)=\left(\begin{array}{cc}
		0 & l(t) \\
		-l(t) & 0
	\end{array}
	\right)\left(\begin{array}{c}
		\nu(t) \\
		\mu(t)
	\end{array}
	\right).$
\end{center}
In addition, there exists a {\color{black}$C^\infty$}
function $\beta(t)$ such that $\dot{\gamma}(t)=\beta(t)\mu(t)$. Hence, $t_0$ is a singular point of $\gamma$ if $\beta(t_0)=0$. The pair $(l(t),\beta(t))$ is called the curvature of the frontal $\gamma$, which is an important invariant of frontals (cf. \cite{FT1}). A point $t_0\in I$ is {\color{black}called}
an {\it inflection point} of $\gamma$ if $l(t_0)=0$. We say that a frontal $\gamma$ is a {\color{black}\it front} if $(l(t),\beta(t))\neq(0,0)$ for any $t\in I$.
In this paper, the curve $\gamma:I\rightarrow\mathbb{R}^2$ {\color{black}used for the definition of} a circle family $\mathcal{C}_{(\gamma,\lambda)}$
is assumed to be a frontal, and the following is adopted as the definition of an envelope created by a circle family.
\begin{definition}\label{d1.2}
	Let $\mathcal{C}_{(\gamma,\lambda)}$ be a circle family.
A {\color{black}$C^\infty$}
mapping $f:I\rightarrow\mathbb{R}^2$ is called an {\it envelope} of
$\mathcal{C}_{(\gamma,\lambda)}$ if the following two conditions are satisfied
{\color{black} for any $t\in I$}.\\
	(1) $f(t)\in \mathcal{C}_{(\gamma(t),\lambda(t))}$.\\
	(2) ${\color{black}\dot{f}}(t)\cdot\big(f(t)-\gamma(t)\big)=0$.
\end{definition}
{\color{black}The following is the key notion for envelopes of circle families.}

\begin{definition}[{\color{black}\cite{WN}}] \label{creative}
	Let $\gamma:I\rightarrow\mathbb{R}^2$ be a frontal with Gauss mapping $\nu:I\rightarrow S^1$ and let $\lambda:I\rightarrow\mathbb{R}_{+}$ be a
	positive {\color{black}$C^\infty$} function.
Then, the circle family $\mathcal{C}_{(\gamma,\lambda)}$ is said to be
{\color{black}\it creative}
if there exists a {\color{black}$C^\infty$}
mapping $\widetilde{\nu}:I\rightarrow S^1$ such that
	\begin{align*}
		\frac{d\lambda}{dt}(t)=-\beta(t)(\widetilde{\nu}(t)\cdot\mu(t)).
	\end{align*}
Set $\cos\theta(t)=-\widetilde{\nu}(t)\cdot\mu(t)${\color{black}.  Then}
the creative condition is equivalent to
{\color{black}the condition} that there exists a {\color{black}$C^\infty$}
funciton $\theta:I\rightarrow\mathbb{R}$ such that
the following identity holds for any $t\in I$.
	\begin{align*}
		\frac{d\lambda}{dt}(t)=\cos\theta(t)\beta(t).
	\end{align*}
	In this case, we have $\widetilde{\nu}(t)=-\cos\theta(t)\mu(t)\pm\sin\theta(t)\nu(t).$
\end{definition}

%
\begin{theorem}[{\color{black}\cite{WN}}]\label{t1.4}
	Let $\gamma: I\to \mathbb{R}^2$
	be a frontal with Gauss mapping $\nu: I\to S^1$ and let
	$\lambda: I\to \mathbb{R}_+$ be a positive {\color{black}$C^\infty$} function.
	Then, the following three hold.
	\begin{enumerate}
		\item[(1)] The circle family $\mathcal{C}_{(\gamma, \lambda)}$
		creates an envelope if and only if $\mathcal{C}_{(\gamma, \lambda)}$ is creative.
		\item[(2)] Suppose that the circle family
		$\mathcal{C}_{(\gamma, \lambda)}$ creates
		an envelope $f: I\to \mathbb{R}^2$.     Then, the created envelope
		$f$ is represented as {\color{black}follows}:
		\[
		f(t) = \gamma(t) + \lambda(t)\widetilde{\nu}(t),
		\]
		where $\widetilde{\nu}: I\to S^1$ is the mapping
		defined in Definition \ref{creative}.
		\item[(3)] Suppose that the circle family
		$\mathcal{C}_{(\gamma, \lambda)}$ creates
		an envelope.
		Then, the number of envelopes created by $\mathcal{C}_{(\gamma, \lambda)}$ is
		characterized as follows.
		\smallskip
		\begin{enumerate}
			\item[(3-i)] The circle family
			$\mathcal{C}_{(\gamma, \lambda)}$
			creates a unique envelope
			if and only if
			the set consisting of $t\in I$ satisfying $\beta(t)\ne 0$  and
			$\frac{d\lambda}{dt}(t)=\pm \beta(t)$ is dense in $I$.
			\item[(3-ii)] There are exactly two distinct envelopes created by
			$\mathcal{C}_{(\gamma, \lambda)}$ if and only if
			the set of $t\in I$ satisfying $\beta(t)\ne 0$ is dense in $I$ and
			there exists at least one $t_0\in I$ such that the strict inequality
			$|\frac{d\lambda}{dt}(t_0)| < |\beta(t_0)|$ holds.
			\item[(3-$\infty$)] There are uncountably many distinct
			envelopes created by
			$\mathcal{C}_{(\gamma, \lambda)}$ if and only if
			the set of $t\in I$ satisfying $\beta(t)\ne 0$ is not dense in $I$.
		\end{enumerate}
	\end{enumerate}
\end{theorem}
By the assertion (2) of Theorem \ref{t1.4}, it is reasonable to call
$\widetilde{\nu}\,$ the \textit{creator} for {\color{black}the} envelope $f$ created by
$\mathcal{C}_{(\gamma, \lambda)}$.
On the other hand, it is well known that the evolute
of a {\color{black}regular} curve
{\color{black}without inflection points}
in the Euclidean plane is not only the locus of cent{\color{black}er}s of the curvature,
but also the envelope of its normal lines.
The involute of a curve is the locus of a point on a piece of taut string as the string
is unwrapped from
the curve. Hence, the involute will vary as the {\color{black}fixed} point varies{\color{black},
and} the curve is
the evolute of any of its involutes. Taking the advantage of envelope theory,
P. Giblin and J. Warder introduced the notion of evolutoids which fills
in the gap between the evolute and the original curve
{\color{black}(see }\cite{GW}{\color{black})}. Each member of the evolutoids is defined
as the envelope of the family of lines, such that each line has a constant angle
with the tangent line of the original curve.
Additionally, the pedal and the contrapedal of a {\color{black}frontal}
are defined as locus of the foot of the perpendicular from a given point to tangents
or normals of the original curve respectively.
Analogous to the evolutoids, S. Izumiya and N. Takeuchi introduced the notion of pedaloid in \cite{IT2}, which fills in the gap between the pedal and the contrapedal.
By using the moving frame $\{\nu(t),\mu(t)\}$, the above associated curves of a frontal $\gamma$ are defined as follows.

\begin{definition}[\cite{FT2}]
	Let $\gamma:I\rightarrow\mathbb{R}^2$ be a frontal with the curvature
$(l,\beta)${\color{black}.   Then,}
the {\color{black}\it involute} of $\gamma$ at $t_0$ is defined {\color{black}as follows.}
	\begin{align*}
		\mathcal{I}nv(\gamma,t_0)(t)=\gamma(t)-\bigg(\int_{t_0}^{t}\beta(u)du\bigg)\mu(t).
	\end{align*}
	If there exists a {\color{black}$C^\infty$} function $\alpha:I\rightarrow\mathbb{R}$
such that $\beta(t)=l(t)\alpha(t)$, the {\color{black}\it evolute}
of the frontal $\gamma$ is defined {\color{black}as follows.}
\begin{align*}
		\mathcal{E}v(\gamma)(t)=\gamma(t)-\alpha(t)\nu(t).
	\end{align*}
\end{definition}

\begin{definition}[\cite{IT2}]
	Let $\gamma:I\rightarrow\mathbb{R}^2$ be a frontal with the curvature $(l,\beta)$. If there exists a {\color{black}$C^\infty$}
function $\alpha:I\rightarrow\mathbb{R}$ such that $\beta(t)=l(t)\alpha(t)$,
the {\color{black}\it $\phi$-evolutoid} of $\gamma$ is defined {\color{black}as follows}.
	\begin{align*} \mathcal{E}v_\gamma[\phi](t)=\gamma(t)-\alpha(t)\sin\phi(\cos\phi\mu(t)+\sin\phi\nu(t)),
	\end{align*}
	where $\phi$ is a fixed angle.
For a fixed point $P${\color{black}$\in\R^2$},
the {\color{black}\it $\phi$-pedaloid} of $\gamma$ relative to $P$
is defined as {\color{black}follows.}
	\begin{align*} \mathcal{P}e_{\gamma,P}[\phi](t)=\gamma(t)+[(P-\gamma(t))\cdot(\sin\phi\mu(t)-\cos\phi\nu(t))](\sin\phi\mu(t)-\cos\phi\nu(t)).
	\end{align*}
	{\color{black}Notice that}, if $\phi=0$ or $\phi=\frac{\pi}{2}$, we have
	\begin{center}
		$\begin{array}{ccc}
			&\mathcal{E}v_\gamma[0](t)=\gamma(t), & \mathcal{P}e_{\gamma,P}[0](t)=\mathcal{CP}e_{\gamma,P}(t),\\
			&\mathcal{E}v_\gamma[\frac{\pi}{2}](t)=\mathcal{E}v(\gamma)(t), & \mathcal{P}e_{\gamma,P}[\frac{\pi}{2}](t)=\mathcal{P}e_{\gamma,P}(t),
		\end{array}$
	\end{center}
	where $\mathcal{P}e_{\gamma,P}(t)$ is the {\color{black}\it pedal} of $\gamma$ relative to $P$, and $\mathcal{CP}e_{\gamma,P}(t)$ is known as the
{\color{black}\it contrapedal} of $\gamma$ relative to $P$.
\end{definition}

\begin{example}\label{e1.1}
	\begin{enumerate}
		\item[(1)]
		Let $\gamma: \mathbb{R}_+\to \mathbb{R}^2$
		be the mapping defined by
		$\gamma(t)=\left(0, t\right)$.     Then, it is clear that $\gamma$ is a frontal.
		Let $\lambda: \mathbb{R}_+\to \mathbb{R}_+$ be the positive function
		defined by $\lambda(t)=t$.
		Then, it is easily seen that the origin $(0,0)$ of the plane $\mathbb{R}^2$ is a unique envelope created by the circle family
		$\mathcal{C}_{(\gamma, \lambda)}$, and $(0,0)$ can be regard as an involute of $\gamma(t)$, or a pedal of $\gamma(t)$ relative to the origin.
		\item[(2)]
		The circle family $\{\{(x, y)\in \mathbb{R}^2\, |\, (x-\cos t)^2+(y-\sin t)^2=1\}\}_{t\in \mathbb{R}}$ creates exactly two
		envelopes: $f_1(t)=(0,0)$ and $f_2(t)=(2\cos t, 2\sin t)$. Moreover, we find that $f_2(t)$ is a pedal of itself relative to $f_1(t)$.
	\end{enumerate}
\end{example}
\noindent
{\color{black}For more details on Example \ref{e1.1}, see Section \ref{section3}.}
\par
\medskip
{\color{black}The main result in this paper is the following Theorem \ref{t1.7}.}
\begin{theorem}\label{t1.7}
	Let $\gamma: I\to \mathbb{R}^2$
	be a frontal with Gauss mapping $\nu: I\to S^1$ and let
	$\lambda:I\to \mathbb{R}_+$ be a positive {\color{black}$C^\infty$} function.
	Then, we have the following:
	\begin{enumerate}
		\item[(1)] Suppose that the circle family $\mathcal{C}_{(\gamma, \lambda)}$
creates an envelope $f: I\to \mathbb{R}^2${\color{black}.   T}hen $f(t)$
is a frontal with Gauss mapping $\widetilde{\nu}:I\rightarrow S^1$ and its curvature is
$$
\big(l(t)\pm\dot{\theta}(t),\lambda(t)[l(t)\pm\dot{\theta}(t)]\pm\beta(t)\sin\theta(t)\big),
$$
{\color{black}where $l, \beta: I\to \R$ is the functions defined in the paragraph
just after Definition \ref{definition1} and $\theta: I\to \R$ is the function defined in
Definition \ref{creative}.}
		\item[(2)] Suppose that the set consisting of $t\in I$ satisfying $\beta(t)\ne 0$ and $\frac{d\lambda}{dt}(t)=\pm\beta(t)$ is dense in $I$, and $f(t)$ is the unique envelope created by the circle family $\mathcal{C}_{(\gamma, \lambda)}$ with Gauss mapping $\widetilde{\nu}:I\rightarrow S^1${\color{black}.   T}hen the following four hold.
		\smallskip
		\begin{enumerate}
			\item[(2-i)] $\gamma(t)$ is the evolute of $f(t)$.
			\item[(2-ii)] $\mathcal{C}_{(\gamma,\cos\phi\lambda)}$ creates two envelopes $f_1(t)$ and $f_2(t)$ such that $$f_1(t)=\mathcal{E}v_f[\phi](t),~~~~~~f_2(t)=\mathcal{E}v_f[\pi-\phi](t),$$ where $\phi
{\color{black}\notin\left\{\frac{\pi}{2}+k\pi\; |\; k\in \Z\right\}}$.
			\item[(2-iii)] Let $\gamma_1(t)=\frac{\gamma(t)+P}{2}$ and $\lambda_1(t)=\|\frac{\gamma(t)-P}{2}\|$ where
$P{\color{black}\notin\{\gamma(t)\; |\; t\in I\}}$ is
a fixed point{\color{black}.  T}hen the circle family $C_{(\gamma_1,\lambda_1)}$ creates envelopes $f_1(t)$ and $f_2(t)$, such that
			$$f_1(t)=P,~~~~~~ f_2(t)=\mathcal{CP}e_{f,P}(t).$$
			\item[(2-iv)] Let $\phi$ be a fixed angle and
{\color{black}let}
$P\notin{\color{black}\left\{\mathcal{E}v_f[\phi+\frac{\pi}{2}](t)\; |\; t\in I\right\}}$
be a fixed point{\color{black}.   T}hen
the circle family $C_{(\gamma_{\color{black}2},\lambda_{\color{black}2})}$
creates envelopes $f_1(t)$ and $f_2(t)$ such that
			$$
f_1(t)=P,~~~~~~f_2(t)=\mathcal{P}e_{f,P}[\phi](t),
          $$
			where
			\begin{align*}
				\gamma_{\color{black}2}(t)=&\frac{\gamma(t)+\lambda(t)\sin\phi\big(-\sin\phi\mu(t)+\cos\phi\nu(t)\big)+P}{2},\\
				\lambda_{\color{black}2}(t)=&\frac{\|\gamma(t)+\lambda(t)\sin\phi\big(-\sin\phi\mu(t)+\cos\phi\nu(t)\big)-P\|}{2}.
			\end{align*}
		\end{enumerate}
		\smallskip
			\item[(3)] Suppose that the set of $t\in I$ satisfying $\beta(t)\ne 0$ and $|\frac{d\lambda}{dt}(t)| \leq |\beta(t)|$ is dense in $I$, and the circle family $\mathcal{C}_{(\gamma, \lambda)}$ creates envelopes $f_1(t)$ and $f_2(t)${\color{black}.    Suppose
moreover that
$f_1(t)$ is a constant vector.    T}hen we have
$$
f_2(t)=\mathcal{P}e_{2\gamma(t)-f_1,f_1}(t).
$$
	\end{enumerate}
\end{theorem}
\bigskip
This paper is organized as follows. The proof of Theorem \ref{t1.7} is given in Section
\ref{section2}.
In Section \ref{section3},
{\color{black} in order to show how}
Theorem \ref{t1.7} is effectively applicable{\color{black}, }
several examples {\color{black}including the
above (1), (2) of Example \ref{e1.1} are given}.
Finally, in Section \ref{section4},
some applications of Theorem \ref{t1.7} are investigated.

\section{Proof of Theorem \ref{t1.7}}\label{section2}
\subsection{Proof of the assertion (1) of Theorem \ref{t1.7}}
\label{subsection2.1}
By Definition \ref{d1.2} and Theorem \ref{t1.4} (2), it follows that $f:I\rightarrow\mathbb{R}^2$ is a frontal since $\frac{df}{dt}(t)\cdot\widetilde{\nu}(t)=0$ for any $t\in I$. We set
\begin{align*}
	\widetilde{\mu}(t)=J(\widetilde{\nu}(t))=\cos\theta(t)\nu(t)\pm\sin\theta(t)\mu(t).
\end{align*}
Then we have
\begin{align*}
	\frac{d}{dt}\widetilde{\nu}(t)=&\cos\theta(t)[l(t)\pm\frac{d\theta}{dt}(t)]\nu(t)\pm\sin\theta(t)[l(t)\pm\frac{d\theta}{dt}(t)]\mu(t)\\
	=&[l(t)\pm\frac{d\theta}{dt}(t)]\widetilde{\mu}(t).
\end{align*}
On the other hand, we calculate that
\begin{align*}
	\frac{df}{dt}(t)=&\beta(t)\mu(t)+\beta(t)\cos\theta(t)\widetilde{\nu}(t)+\lambda(t)[l(t)\pm\frac{d\theta}{dt}(t)]\widetilde{\mu}(t)\\
	=&\beta(t)\mu(t)+\beta(t)\cos\theta(t)[-\cos\theta(t)\mu(t)\pm\sin\theta(t)\nu(t)]+\lambda(t)[l(t)\pm\frac{d\theta}{dt}(t)]\widetilde{\mu}(t)\\
	=&\pm\beta(t)\sin\theta(t)\widetilde{\mu}(t)+\lambda(t)[l(t)\pm\frac{d\theta}{dt}(t)]\widetilde{\mu}(t)\\
	=&\big(\lambda(t)[l(t)\pm\frac{d\theta}{dt}(t)]\pm\beta(t)\sin\theta(t)\big)\widetilde{\mu}(t).
\end{align*}
Thus, the curvature of $f$ is $\big(l(t)\pm\frac{d\theta}{dt}(t),\lambda(t)[l(t)\pm\frac{d\theta}{dt}(t)]\pm\beta(t)\sin\theta(t)\big)$.
\hfill $\Box$

\subsection{Proof of the assertion (2-i) of Theorem \ref{t1.7}}
\label{subsection2.2}
By Theorem \ref{t1.4} (3-i), $\mathcal{C}_{(\gamma, \lambda)}$
creates a unique envelope if and only if the set consisting of $t\in I$ satisfying
$\frac{d\lambda}{dt}(t)=\pm \beta(t)\ne 0$ is dense in $I$.
Under considering continuity of
the function $t\mapsto \widetilde{\nu}(t)\cdot\mu(t)$, we have
$\widetilde{\nu}(t)\cdot\mu(t)=\pm 1$ for any $t\in I$. Then,
{\color{black}it follows that}
$\mathcal{C}_{(\gamma, \lambda)}$ creates a unique envelope
{\color{black}if and only if}
$\theta(t)=k\pi$ for any $t\in I$.
{\color{black}Thus, by the assertion (1)},
we have $\lambda(t)l_f(t)=\beta_f(t)$.
By Theorem \ref{t1.4} (2), $f(t)$ can be represented by
\begin{align*}
	f(t)=\gamma(t)+\lambda(t)\widetilde{\nu}(t).
\end{align*}
According to the definition of evolutes of frontals, we have
\begin{align*}
	\mathcal{E}v(f)(t)=f(t)-\lambda(t)\widetilde{\nu}(t)=\gamma(t).
\end{align*}
\hfill $\Box$
\subsection{Proof of the assertion (2-ii) of Theorem \ref{t1.7}}
\label{subsection2.3}
By the assertion (2-i), we have $\widetilde{\nu}(t)=-\mu(t)$ and $\frac{d\lambda}{dt}(t)=\beta(t)$, or $\widetilde{\nu}(t)=\mu(t)$ and $\frac{d\lambda}{dt}(t)=-\beta(t)$. If $\widetilde{\nu}(t)=-\mu(t)$, $\frac{d\lambda}{dt}(t)=\beta(t)$ and $\phi\neq\frac{\pi}{2}+k\pi$, consider the circle family $\mathcal{C}_{(\gamma,\cos\phi\lambda)}$
{\color{black}.   It is easily seen that}
$\mathcal{C}_{(\gamma,\cos\phi\lambda)}$ is creative and the following equalities hold.
\begin{align*}
	\cos\phi\frac{d\lambda}{dt}(t)=\cos\phi\beta(t),~~~ \widetilde{\nu}_{f_1}(t)=-\cos\phi\mu(t)-\sin\phi\nu(t),~~~
	\widetilde{\nu}_{f_2}(t)=-\cos\phi\mu(t)+\sin\phi\nu(t).
\end{align*}
Since $f_1(t)=\gamma(t)+\lambda(t)\cos\phi\widetilde{\nu}_{f_1}(t)$ and $f_2(t)=\gamma(t)+\lambda(t)\cos\phi\widetilde{\nu}_{f_2}(t)$, we have
\begin{align*}
	f_1(t)=&\gamma(t)+\lambda(t)\cos\phi(-\cos\phi\mu(t)-\sin\phi\nu(t))\\
	=&f(t)-\lambda(t)\widetilde{\nu}(t)+\lambda(t)\cos\phi(-\cos\phi\mu(t)-\sin\phi\nu(t))\\
	=&f(t)-\lambda(t)\widetilde{\nu}(t)+\lambda(t)\cos\phi(\cos\phi\widetilde{\nu}(t)-\sin\phi\widetilde{\mu}(t))\\
	=&f(t)-\lambda(t)\sin\phi(\cos\phi\widetilde{\mu}(t)+\sin\phi\widetilde{\nu}(t))
\end{align*}
and
\begin{align*}
	f_2(t)=&\gamma(t)+\lambda(t)\cos\phi(-\cos\phi\mu(t)+\sin\phi\nu(t))\\
	=&f(t)-\lambda(t)\widetilde{\nu}(t)+\lambda(t)\cos\phi(-\cos\phi\mu(t)+\sin\phi\nu(t))\\
	=&f(t)-\lambda(t)\widetilde{\nu}(t)+\lambda(t)\cos\phi(\cos\phi\widetilde{\nu}(t)+\sin\phi\widetilde{\mu}(t))\\
	=&f(t)-\lambda(t)\sin\phi(-\cos\phi\widetilde{\mu}(t)+\sin\phi\widetilde{\nu}(t))\\
	=&f(t)-\lambda(t)\sin(\pi-\phi)(\cos(\pi-\phi)\widetilde{\mu}(t)+\sin(\pi-\phi)\widetilde{\nu}(t)).
\end{align*}
By the assumption, it follows that $\lambda(t)l_f(t)=\beta_f(t)$ for any $t\in I$. According to the definition of evolutoids of frontals, we obtain
$$f_1(t)=\mathcal{E}v_f[\phi](t),~~~~~~f_2(t)=\mathcal{E}v_f[\pi-\phi](t).$$
For the case of $\widetilde{\nu}(t)=\mu(t)$ and $\frac{d\lambda}{dt}(t)=-\beta(t)$, we can prove it in the same way.
\hfill $\Box$
\subsection{Proof of the assertion (2-iii) of Theorem \ref{t1.7}}
\label{subsection2.4}
The proof of the assertion (2-iv) given in Subsection 2.5
proves the assertion (2-iii) as well.
\hfill $\Box$
\subsection{Proof of the assertion (2-iv) of Theorem \ref{t1.7}}
\label{subsection2.5}
By the assumption, it follows that
\begin{align*}
	\big(\gamma_1(t)-P\big)\cdot\big(\gamma_1(t)-P\big)=\lambda_1^2(t)
\end{align*}
always holds for the circle family $\mathcal{C}_{(\gamma_1, \lambda_1)}$.
{\color{black}   T}hen $\mathcal{C}_{(\gamma_1, \lambda_1)}$ is creative and $P$
is an envelope of $\mathcal{C}_{(\gamma_1, \lambda_1)}$.
Moreover, by the proof of {\color{black}the} assertion (2-ii), we have
\begin{align*}
	\mathcal{E}v_f[\phi+\frac{\pi}{2}](t)=\gamma(t)+\lambda(t)\sin\phi\big(-\sin\phi\mu(t)+\cos\phi\nu(t)\big).
\end{align*}
As $P\notin\mathcal{E}v_f[\phi+\frac{\pi}{2}](t)$, it ensures that $\lambda_1(t)>0$ for any $t\in I.$
Without losing generality, we {\color{black}may} choose the origin as $P$.
Then the envelope $\widetilde{f}(t)$ of $\mathcal{C}_{(\gamma_1,\lambda_1)}$ satisfies
\begin{align*}
	\big(\widetilde{f}(t)-\gamma_1(t)\big)\cdot\big(\widetilde{f}(t)-\gamma_1(t)\big)=\gamma_1(t)\cdot \gamma_1(t).
\end{align*}
{\color{black}Just from this equation,}
we have
\begin{align}\label{eq2.1}
	\widetilde{f}(t)\cdot\big(\widetilde{f}(t)-2\gamma_1(t)\big)=0.
\end{align}
{\color{black}In addition, s}ince $\widetilde{f}(t)$ is the envelope,
the following equality holds.
\begin{align}\label{eq2.2}
	\frac{d\widetilde{f}}{dt}(t)\cdot\big(\widetilde{f}(t)-\gamma_1(t)\big)=0.
\end{align}
{\color{black}From (\ref{eq2.2}) and differentiating (\ref{eq2.1})},
we obtain that
\begin{align*}
	\widetilde{f}(t)\cdot\frac{d\gamma_1}{dt}(t)=0.
\end{align*}
{\color{black}Since}
$\gamma_1(t)$ is a frontal with the Gauss mapping $\nu_{\gamma_1}:I\rightarrow S^1$ defined by $$\nu_{\gamma_1}(t)=-\sin\phi\mu(t)+\cos\phi\nu(t),$$
{\color{black}it follows that}
there exists a
{\color{black}$C^\infty$}
function $\delta(t)$ such that $\widetilde{f}(t)=\delta(t)\nu_{\gamma_1}(t)$.
{\color{black}This implies}
that $\delta(t)=0$ or $\delta(t)=2\gamma_1(t)\cdot\nu_{\gamma_1}(t)$ since $\widetilde{f}(t)\cdot\big(\widetilde{f}(t)-2\gamma_1(t)\big)=0$. This means
\begin{align*}
\widetilde{f}(t)=0~~~~~ \mbox{or} ~~~~~ \widetilde{f}(t)=\big(2\gamma_1(t)\cdot\nu_{\gamma_1}(t)\big)\nu_{\gamma_1}(t).
\end{align*}

On the other hand, by the proof of the assertion (2-i), we have $\widetilde{\nu}(t)=\pm\mu(t)$ and $\widetilde{\mu}(t)=\mp\nu(t).$
{\color{black}Here, two double signs $\pm, \mp$ should be read in same order}.
If $\widetilde{\nu}(t)=-\mu(t)$ and $\widetilde{\mu}(t)=\nu(t),$ we have the following:
\begin{align*}
	&\big(2\gamma_1(t)\cdot\nu_{\gamma_1}(t)\big)\nu_{\gamma_1}(t)\\
	=&\big[\big(\gamma(t)+\lambda(t)\sin\phi(-\sin\phi\mu(t)+\cos\phi\nu(t))\big)\cdot(-\sin\phi\mu(t)+\cos\phi\nu(t))\big](-\sin\phi\mu(t)+\cos\phi\nu(t))\\
	=&\big[\gamma(t)\cdot(\sin\phi\widetilde{\nu}(t)+\cos\phi\widetilde{\mu}(t))\big](\sin\phi\widetilde{\nu}(t)+\cos\phi\widetilde{\mu}(t))+\lambda(t)\sin\phi(\sin\phi\widetilde{\nu}(t)+\cos\phi\widetilde{\mu}(t))\\
	=&\big[(f(t)-\lambda(t)\widetilde{\nu}(t))\cdot(\sin\phi\widetilde{\nu}(t)+\cos\phi\widetilde{\mu}(t))\big](\sin\phi\widetilde{\nu}(t)+\cos\phi\widetilde{\mu}(t))+\lambda(t)\sin\phi(\sin\phi\widetilde{\nu}(t)+\cos\phi\widetilde{\mu}(t))\\
	=&\big[f(t)\cdot(\sin\phi\widetilde{\nu}(t)+\cos\phi\widetilde{\mu}(t))\big](\sin\phi\widetilde{\nu}(t)+\cos\phi\widetilde{\mu}(t))\\
	=&\mathcal{P}e_{f,0}[\phi](t).
\end{align*}
{\color{black}In the case}
$\widetilde{\nu}(t)=\mu(t)$ and $\widetilde{\mu}(t)=-\nu(t)$,
{\color{black}it is easily seen that}
$\big(2\gamma_1(t)\cdot\nu_{\gamma_1}(t)\big)\nu_{\gamma_1}(t)=\mathcal{P}e_{f,0}[\phi](t)$ {\color{black}can be shown} in the same way.
\hfill $\Box$

\subsection{Proof of the assertion (3) of Theorem \ref{t1.7}}
\label{subsection2.6}
Under considering continuity of the functions $\frac{d\lambda}{dt}(t)$ and $\beta(t)$,
{\color{black}it is easily seen}
that $\mathcal{C}_{(\gamma, \lambda)}$ is creative.
By Definition \ref{d1.2}, {\color{black}the assertion \lq\lq}an envelope
of a circle family is a point{\color{black}\rq\rq\; is equivalent to the assertion
\lq\lq}all circles of the family pass through the point{\color{black}\rq\rq}.
Without losing generality, {\color{black}we may} choose the origin
as the constant vector $f_1${\color{black}.
T}hen the envelope $f(t)$ of $\mathcal{C}_{(\gamma, \lambda)}$ satisfies
\begin{align*}
	\big(f(t)-\gamma(t)\big)\cdot\big(f(t)-\gamma(t)\big)=\gamma(t)\cdot\gamma(t).
\end{align*}
{\color{black}This implies}
\begin{align*}
	f(t)\cdot\big(f(t)-2\gamma(t)\big)=0.
\end{align*}
{\color{black}On the other hand, s}ince $f(t)$ is the envelope,
{\color{black}it follows}
$$\frac{df}{dt}(t)\cdot\big(f(t)-\gamma(t)\big)=0.$$
{\color{black}Thus, by the similar way as}
the proof of assertion (2-iv), we have $f(t)=0$ or $f(t)=\big(2\gamma(t)\cdot\nu(t)\big)\nu(t).$
{\color{black}Therefore, by} the assumption $f_1(t)=0$, we have
$$f_2(t)=\big(2\gamma(t)\cdot\nu(t)\big)\nu(t)=\mathcal{P}e_{2\gamma(t),0}(t).$$
\hfill $\Box$
\section{Examples}\label{section3}
\begin{example}
	We examine (1) of Example \ref{e1.1} by applying Theorem \ref{t1.7}. In (1) of Example \ref{e1.1}, $\gamma: \mathbb{R}_+\to \mathbb{R}^2$
	is given by $\gamma(t)=\left(0, t\right)$.
	Thus, if we define the unit vector $\nu(t)=(1, 0)$,
	$\nu: \mathbb{R}_+\to S^1$ gives the Gauss mapping of $\gamma$.
	By definition, $\mu(t)=(0,-1)$ and thus we have
	$$\beta(t)=\frac{d\gamma}{d t}(t)\cdot\mu(t)=-1.$$
	On the other hand, the radius function
	$\lambda: \mathbb{R}_+\to \mathbb{R}_+$ has the form
	$\lambda(t)=t$ in this example.
	Thus, the set consisting of $t\in\mathbb{R}_+$ satisfying $\beta(t)\ne 0$  and $\frac{d\lambda}{dt}(t)=-\beta(t)$ is dense in $\mathbb{R}_+$. By (3-i) of Theorem \ref{t1.4}, $(0,0)$ is the unique envelope of the circle family $\mathcal{C}_{(\gamma, \lambda)}$ with the Gauss mapping $\widetilde{\nu}$ defined by $\widetilde{\nu}=(0,-1)$. Thus, by (2-i) of Theorem \ref{t1.7} and taking $\alpha_{f}(t)=\lambda(t)$ such that $\alpha_{f}(t)l_{f}(t)=\beta_{f}(t)$, $\gamma(t)$ is an evolute of $(0,0)$.
	Moreover, we can use the assertion (3) of Theorem \ref{t1.7} to interpret (1) of Example \ref{e1.1}, in which we have $f_1(t)=f_2(t)=(0,0)$ and $\gamma(t)$ represents the same curve as $2\gamma(t)$. Then $(0,0)$ is also a pedal of $\gamma(t)$ relative to the origin.
\end{example}

\begin{example}
	Theorem \ref{t1.7} can be applied also to (2) of Example \ref{e1.1} as follows. In this example, $\gamma(t)=(\cos t,\sin t)$ and $\lambda(t)=1.$ Thus, we may take $\nu(t)=(\cos t,\sin t)$ and $\mu(t)=(-\sin t, \cos t)$. We have $\beta(t)=\frac{d\gamma}{dt}(t)\cdot\mu(t)=1$. Since the
	radius function $\lambda$ is constant function, the created condition
	\[
	\frac{d\lambda}{dt}(t) =
	-\beta(t)\left(\widetilde{\nu}(t)\cdot \mu(t)\right)
	\]
	becomes simply
	\[
	0 =
	-\left(\widetilde{\nu}(t)\cdot (-\sin t, \cos t)\right)
	\]
	in this case. Thus, for any $t\in\mathbb{R}$,
	the creat{\color{black}ive} condition is satisfied if and only if
	$\widetilde{\nu}(t)=\pm\nu(t)$. Hence, the parametrization of the created envelope is
	\[
	f(t)=\gamma(t)+\lambda(t)\widetilde{\nu}(t)
	=(\cos t\pm\cos t,\sin t\pm\sin t).
	\]
	Since the set of $t\in I$ satisfying $|\frac{d\lambda}{dt}(t)| \leq |\beta(t)|$ is dense in $\mathbb{R}$, by (4) of Theorem \ref{t1.7}, $f_2(t)=(2\cos t, 2\sin t)$ is a pedal of itself relative to the origin.
\end{example}

{\color{black}More} examples are provided to demonstrate Theorem \ref{t1.7} as follows.
\begin{example}\label{e3.1}
	Let $\gamma:\mathbb{R}\rightarrow\mathbb{R}^2$ be the mapping defined by
$\gamma(t)={\color{black}\left(-4t^3,\; 3t^2+\frac{1}{2}\right)}$
and let $\lambda:\mathbb{R}\rightarrow\mathbb{R}$ be the function defined by $\lambda(t)=\frac{1}{2}(1+4t^2)^{\frac{3}{2}}$. Defining the mapping $\nu:\mathbb{R}\rightarrow S^1$ by $\nu(t)=\frac{1}{\sqrt{1+4t^2}}(1,2t)$ clarifies that the mapping $\gamma$ is a frontal. Then we have
	\begin{align*}
		\mu(t)=J(\nu(t))=\frac{1}{\sqrt{1+4t^2}}(-2t,1),~~~~l(t)=\frac{2}{1+4t^2},~~~~\beta(t)=6t\sqrt{1+4t^2}.
	\end{align*}
	The mapping $\theta:\mathbb{R}\rightarrow\mathbb{R}$ defined by $\theta(t)=0$ satisfies the creative condition $\frac{d\lambda}{dt}(t)=\cos\theta(t)\beta(t)$.
{\color{black}   Thus, we have}
	\begin{align*}
		\widetilde{\nu}(t)=-\cos\theta(t)\mu(t)\pm\sin\theta(t)\nu(t)=-\frac{1}{\sqrt{1+4t^2}}(-2t,1).
	\end{align*}
	By the assertion (2) of Theorem \ref{t1.4}, we obtain the unique envelope of $\mathcal{C}_{(\gamma, \lambda)}$ as follows.
	\begin{align*}
		f(t)=\gamma(t)+\lambda(t)\widetilde{\nu}(t)=(t,t^2).
	\end{align*}
	Since the set consisting of $t\in I$ satisfying $\beta(t)\ne 0$  and $\frac{d\lambda}{dt}(t)=\pm \beta(t)$ is dense in $\mathbb{R}$, by the assertion (2-i) of Theorem \ref{t1.7}, $\gamma(t)$ is the evolute of $f(t)$.\par
	On the other hand, by the assertion (1) of Theorem \ref{t1.7}, $f:\mathbb{R}\rightarrow\mathbb{R}^2$ is a frontal with the Gauss mapping $\widetilde{\nu}$ and
the curvature {\color{black}is
$\left(\frac{2}{1+4t^2},\sqrt{1+4t^2}\right)$}.
Thus, the evolute of $f$ is parametrized as follows.
	\begin{align*}
		\mathcal{E}v(f)(t)=(t,t^2)+\frac{(1+4t^2)^{\frac{3}{2}}}{2}\cdot\frac{1}{\sqrt{1+4t^2}}(-2t,1)={\color{black}\left(-4t^3,3t^2+\frac{1}{2}\right)}.
	\end{align*}
	It also follows that $\gamma(t)$ is the evolute of $f(t)$. The circle family $\mathcal{C}_{(\gamma, \lambda)}$ and its envelope $f(t)$ are
	depicted in Figure 3.
	\begin{figure}[htbp]
		\begin{center}
			\small
			\begin{tabular}{cc}
				\includegraphics[width=70mm]{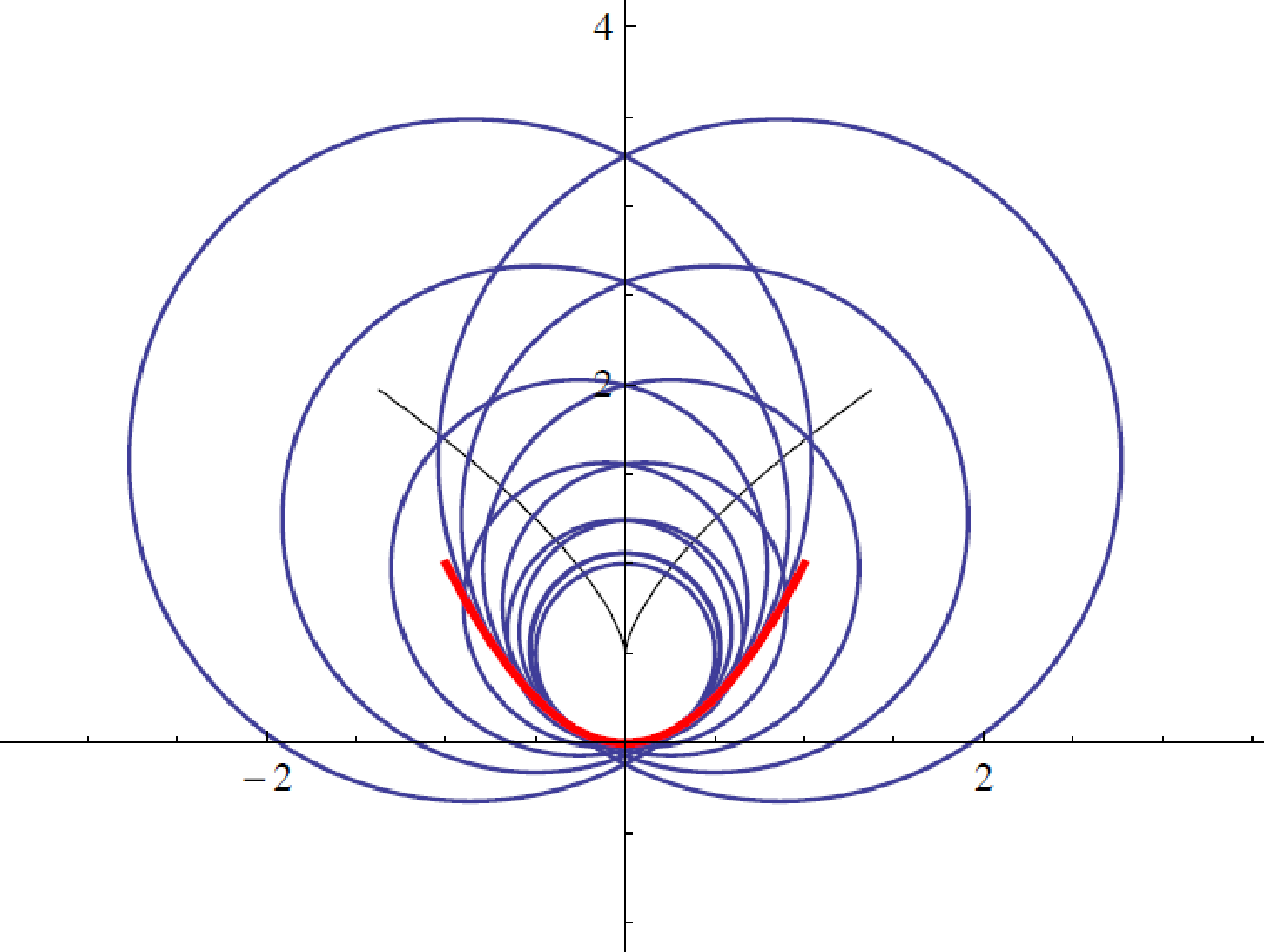}\\
				FIGURE 3. The circle family $\mathcal{C}_{(\gamma, \lambda)}$ , the loci of centers $\gamma(t)$ (thin) and the envelope $f(t)$ (thick).
			\end{tabular}
		\end{center}
	\end{figure}
\end{example}


\begin{example}
	Let $\gamma(t)={\color{black}\left(-4t^3,3t^2+\frac{1}{2}\right)}$ and $\lambda(t)=\frac{1}{2}(1+4t^2)^{\frac{3}{2}}$ which are {\color{black}the} same as
{\color{black}$\gamma, \lambda$ in} Example \ref{e3.1}. Then we have
	\begin{align*}
		&\nu(t)=\frac{1}{\sqrt{1+4t^2}}(1,2t),~~~~~~l(t)=\frac{2}{1+4t^2},\\
		&\mu(t)=\frac{1}{\sqrt{1+4t^2}}(-2t,1),~~~~~~\beta(t)=6t\sqrt{1+4t^2}.
	\end{align*}
	By Example \ref{e3.1}, the circle family $\mathcal{C}_{(\gamma,\lambda)}$ creates only one envelope $f(t)=(t,t^2)$ and the curvature of $f$ is
${\color{black}\left(\frac{2}{1+4t^2},\sqrt{1+4t^2}\right)}$.
	We consider the circle family
$\mathcal{C}_{(\gamma,\cos\frac{\pi}{4}\lambda)}${\color{black}.
B}y the assertion (3-ii) of Theorem \ref{t1.4},
$C_{(\gamma,\cos\frac{\pi}{4}\lambda)}$ creates
two envelopes $f_1(t)$ and $f_2(t)$. {\color{black}By calculations, we have}
	\begin{align*}
		\widetilde{\nu}_{f_1}(t)=-\frac{1}{\sqrt{2}\sqrt{1+4t^2}}(1-2t,1+2t),~~~~
		\widetilde{\nu}_{f_2}(t)=\frac{1}{\sqrt{2}\sqrt{1+4t^2}}(1+2t,-1+2t).
	\end{align*}
	Therefore, we have
	\begin{align*}
		f_1(t)=
{\color{black}\left(-2t^3-t^2+\frac{t}{2}-\frac{1}{4},\;
-2t^3+2t^2-\frac{t}{2}+\frac{1}{4}\right)},~~~~
		f_2(t)={\color{black}\left(-2t^3+t^2+\frac{t}{2}+\frac{1}{4},\;
2t^3+2t^2+\frac{t}{2}+\frac{1}{4}\right)}.
	\end{align*}
	Since the set consisting of $t\in\mathbb{R}$ satisfying $\beta(t)\ne 0$  and $\frac{d\lambda}{dt}(t)=\pm \beta(t)$ is dense in $\mathbb{R}$, by the assertion (2-ii) of Theorem \ref{t1.7}, we obtain that $f_1(t)=\mathcal{E}v_f[\frac{\pi}{4}](t)$ and $f_2(t)=\mathcal{E}v_f[\frac{3\pi}{4}](t)$. It can also be examined that $f_1(t)=\mathcal{E}v_f[\frac{\pi}{4}](t)$ and $f_2(t)=\mathcal{E}v_f[\frac{3\pi}{4}](t)$
by the definition of frontal. The circle family $\mathcal{C}_{(\gamma,\cos\frac{\pi}{4}\lambda)}$ and its envelopes are depicted in Figure 4.
	\begin{figure}[htbp]
		\begin{center}
			\small
			\begin{tabular}{cc}
				\includegraphics[width=75mm]{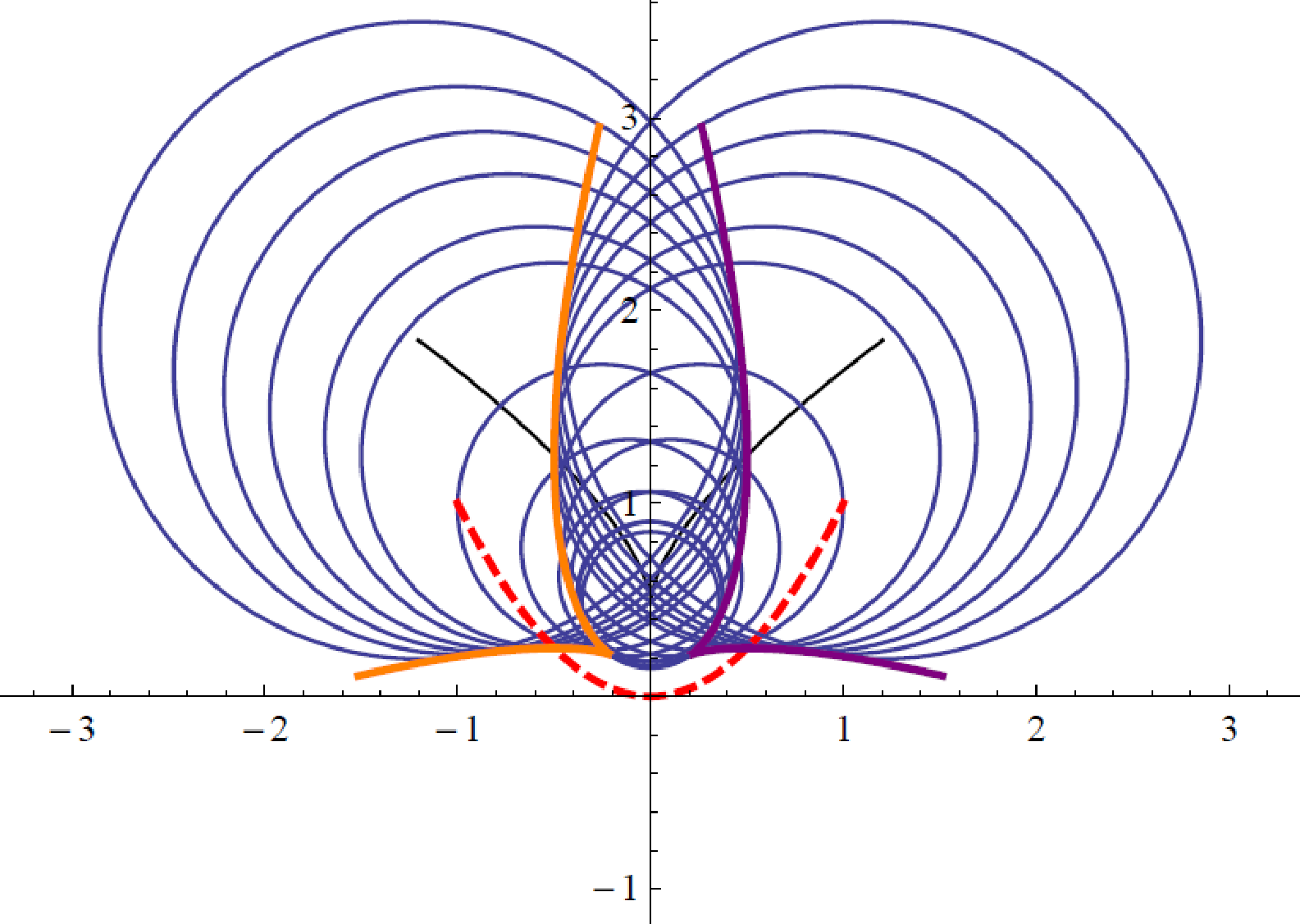} \\
				FIGURE 4. The circle family $\mathcal{C}_{(\gamma,\cos\frac{\pi}{4}\lambda)}$, the loci of centers $\gamma(t)$ (thin), the curve $f(t)$ (dashed)\\
 and the envelopes $f_1(t)$ (left thick), $f_2(t)$ (right thick).
			\end{tabular}
		\end{center}
	\end{figure}
\end{example}

\begin{example}
	Let $\gamma:\mathbb{R}_{+}\rightarrow\mathbb{R}^2$ be the mapping defined by $\gamma(t)=(-4t^3,3t^2)$ and let $\lambda:\mathbb{R}_{+}\rightarrow\mathbb{R}_{+}$ be the function defined by $\lambda(t)=\frac{(1+4t^2)^{\frac{3}{2}}}{2}$.
	We calculate that
	\begin{align*}
		&\nu(t)=\frac{1}{\sqrt{1+4t^2}}(1,2t),~~~~~~l(t)=\frac{2}{1+4t^2},\\
		&\mu(t)=\frac{1}{\sqrt{1+4t^2}}(-2t,1),~~~~~~\beta(t)=6t\sqrt{1+4t^2}.
	\end{align*}
By similar analysis {\color{black}as the one given in} Example \ref{e3.1},
{\color{black}it is relatively easy to show} that the circle family $\mathcal{C}_{(\gamma,\lambda)}$ creates {\color{black}the} unique envelope $f(t)={\color{black}\left(t,t^2-\frac{1}{2}\right)}$.
	We consider the circle family $C_{(\gamma_1,\lambda_1)}$ where
\begin{align*}
		\gamma_1(t)=\frac{\gamma(t)}{2},~~~~\lambda_1(t)=t^2\sqrt{4t^2+\frac{9}{4}}=
{\color{black}\left\Vert\frac{\gamma(t)}{2}\right\Vert}.
\end{align*}
	By the assertion (3-ii{\color{black})} of Theorem \ref{t1.4}, the circle family $C_{(\gamma_1,\lambda_1)}$ creates two envelopes $f_1(t)$ and $f_2(t)$. We calculate that
	\begin{align*}
		\widetilde{\nu}_{f_1}(t)={\color{black}\left(\frac{4t}{\sqrt{9+16t^2}},-\frac{3}{\sqrt{9+16t^2}}\right)},~~~~
		\widetilde{\nu}_{f_2}(t)={\color{black}\left(\frac{8t+16t^3}{(1+4t^2)\sqrt{9+16t^2}},-\frac{3+4t^2}{(1+4t^2)\sqrt{9+16t^2}}\right)}.
	\end{align*}
	Then, the envelopes of $C_{(\gamma_1,\lambda_1)}$ are parametrized as follows.
	\begin{align*}
		f_1(t)=\gamma_1(t)+\lambda_1(t)\widetilde{\nu}_{f_1}(t)=(0,0),~~~~\;
		f_2(t)=\gamma_1(t)+\lambda_1(t)\widetilde{\nu}_{f_2}(t)=
{\color{black}\left(\frac{2t^3}{1+4t^2},\frac{4t^4}{1+4t^2}\right)}.
	\end{align*}
	Since $f_1(t)=(0,0)$ and the set consisting of $t\in\mathbb{R}_{+}$ satisfying $\beta(t)\ne 0$  and $\frac{d\lambda}{dt}(t)=\pm \beta(t)$ is dense in $\mathbb{R}_{+}$, by the assertion (2-iii) of Theorem \ref{t1.7}, we have $f_2(t)=\mathcal{CP}e_{f,0}(t)$. On the other hand, by the definition of contrapedals of frontals,
	we can examine that $f_2(t)$ is the contrapedal curve of $f(t)$ relative to the origin (see Figure 5).\par
	In order to illustrate the assertion (2-iv) of Theorem \ref{t1.7}, we consider another circle family $C_{(\gamma_2,\lambda_2)}$ where
	\begin{align*}
		\gamma_2(t)=&\frac{\gamma(t)+\lambda(t)\sin\frac{\pi}{4}\big(-\sin\frac{\pi}{4}\mu(t)+\cos\frac{\pi}{4}\nu(t)\big)}{2}\\
		=&
{\color{black}\left(-t^3+\frac{t^2}{2}+\frac{t}{4}+\frac{1}{8},\; t^3+t^2+\frac{t}{4}
-\frac{1}{8}\right)},\\
		\lambda_2(t)=&
{\color{black}\left\Vert\frac{\gamma(t)+\lambda(t)\sin\frac{\pi}{4}
\big(-\sin\frac{\pi}{4}\mu(t)+\cos\frac{\pi}{4}\nu(t)\big)}{2}\right\Vert}\\
		=&\frac{\sqrt{2}}{8}\sqrt{(1-4t+8t^2)(1+4t+8t^2+8t^3+8t^4)}.
	\end{align*}
	By calculation, we have
	\begin{align*}
		\mu_{\gamma_2}(t)=&\frac{1}{\sqrt{2+8t^2}}(1-2t,1+2t),~~~~\nu_{\gamma_2}(t)=\frac{1}{\sqrt{2+8t^2}}(1+2t,2t-1),\\
		\beta_{\gamma_2}(t)=&\frac{\sqrt{2+8t^2}(1+6t)}{4},~~~~\cos\theta(t)=\frac{2t^2(3+2t+8t^2)}{\sqrt{(1+4t^2)(1+8t^3+40t^4+32t^5+64t^6)}}.
	\end{align*}
	By the assertion (3-ii{\color{black})} of Theorem \ref{t1.4}, the circle family $C_{(\gamma_2,\lambda_2)}$ creates two envelopes $f_3(t)$ and $f_4(t)$. And we calculate that
	\begin{align*}
		\widetilde{\nu}_{f_3}(t)=&\bigg(\frac{-1-2t-4t^2+8t^3}{\sqrt{2}\sqrt{1+8t^3+40t^4+32t^5+64t^6}},-\frac{-1+2t+8t^2+8t^3}{\sqrt{2}\sqrt{1+8t^3+40t^4+32t^5+64t^6}}\bigg),\\
		\widetilde{\nu}_{f_4}(t)=&\bigg(\frac{1+2t-4t^2+16t^3+32t^5}{\sqrt{2}(1+4t^2)\sqrt{1+8t^3+40t^4+32t^5+64t^6}},-\frac{1-2t+8t^2+16t^3+16t^4+32t^5}{\sqrt{2}(1+4t^2)\sqrt{1+8t^3+40t^4+32t^5+64t^6}}\bigg).
	\end{align*}
	Therefore, the envelopes of $C_{(\gamma_2,\lambda_2)}$ are parametrized as follows.
	\begin{align*}
		f_3(t)=(0,0),~~~~
		f_4(t)=
{\color{black}\left(
\frac{1+2t+2t^2+8t^3+8t^4}{4+16t^2},\frac{-1+2t-2t^2+8t^4}{4+16t^2}\right)}.
	\end{align*}
	Since $f_3(t)=(0,0)$ and the set consisting of $t\in\mathbb{R}_{+}$ satisfying $\beta(t)\ne 0$  and $\frac{d\lambda}{dt}(t)=\pm \beta(t)$ is dense in $\mathbb{R}_{+}$, by the assertion (2-iv) of Theorem \ref{t1.7},   {\color{black}it follows that} $f_4(t)=\mathcal{P}e_{f,\bm{0}}[\frac{\pi}{4}](t)$. Moreover,
{\color{black}it is not difficult to show}
that $f_4(t)$ is the $\frac{\pi}{4}$-pedaloid of $f(t)$ relative to the origin from the definition of pedaioid of frontals (see Figure 6).
	\begin{figure}[htbp]
		\begin{center}
			\small
			\begin{tabular}{cc}
				\includegraphics[width=70mm]{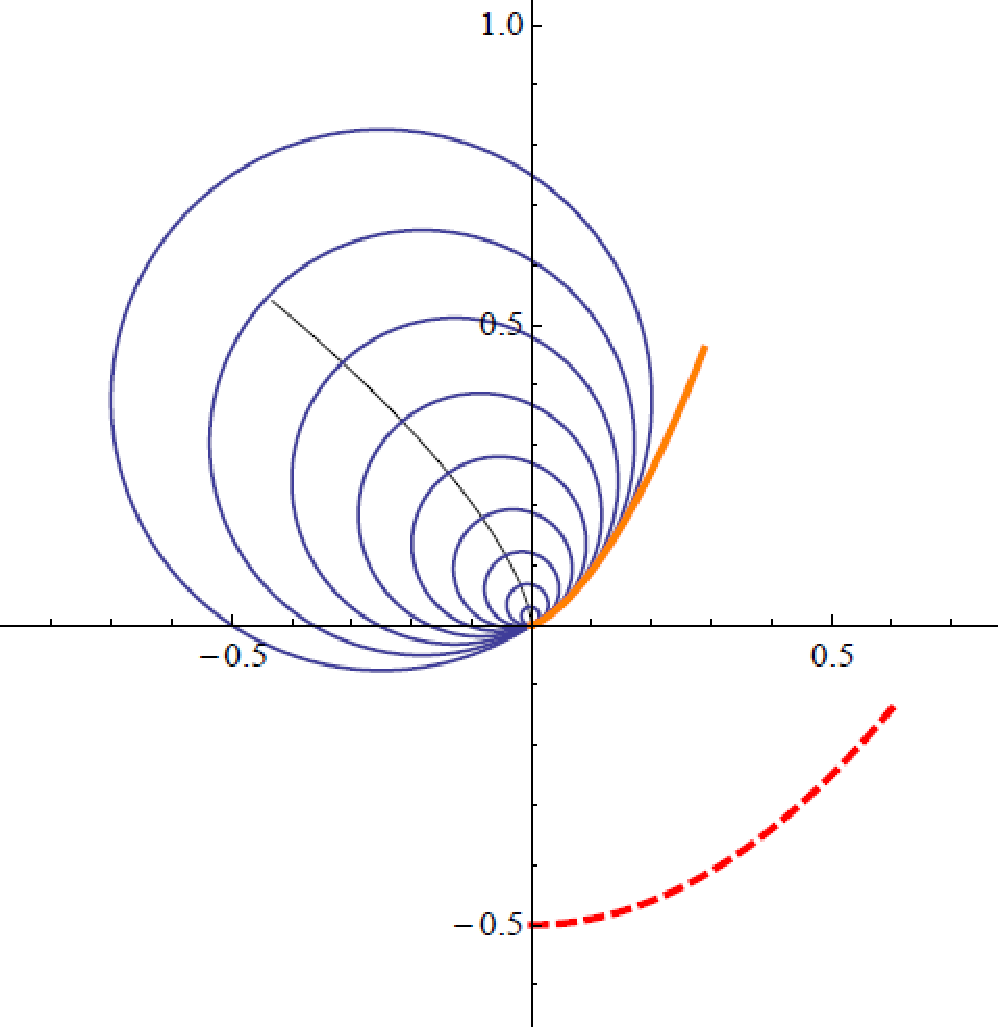} \\
				FIGURE 5. The circle family $C_{(\gamma_1,\lambda_1)}$, the loci of centers $\gamma_1(t)$ (thin),\\ the envelope $f_2(t)$ (thick) and the curve $f(t)$ (dashed).
			\end{tabular}
		\end{center}
	\end{figure}
	\begin{figure}[htbp]
		\begin{center}
			\small
			\begin{tabular}{cc}
				\includegraphics[width=70mm]{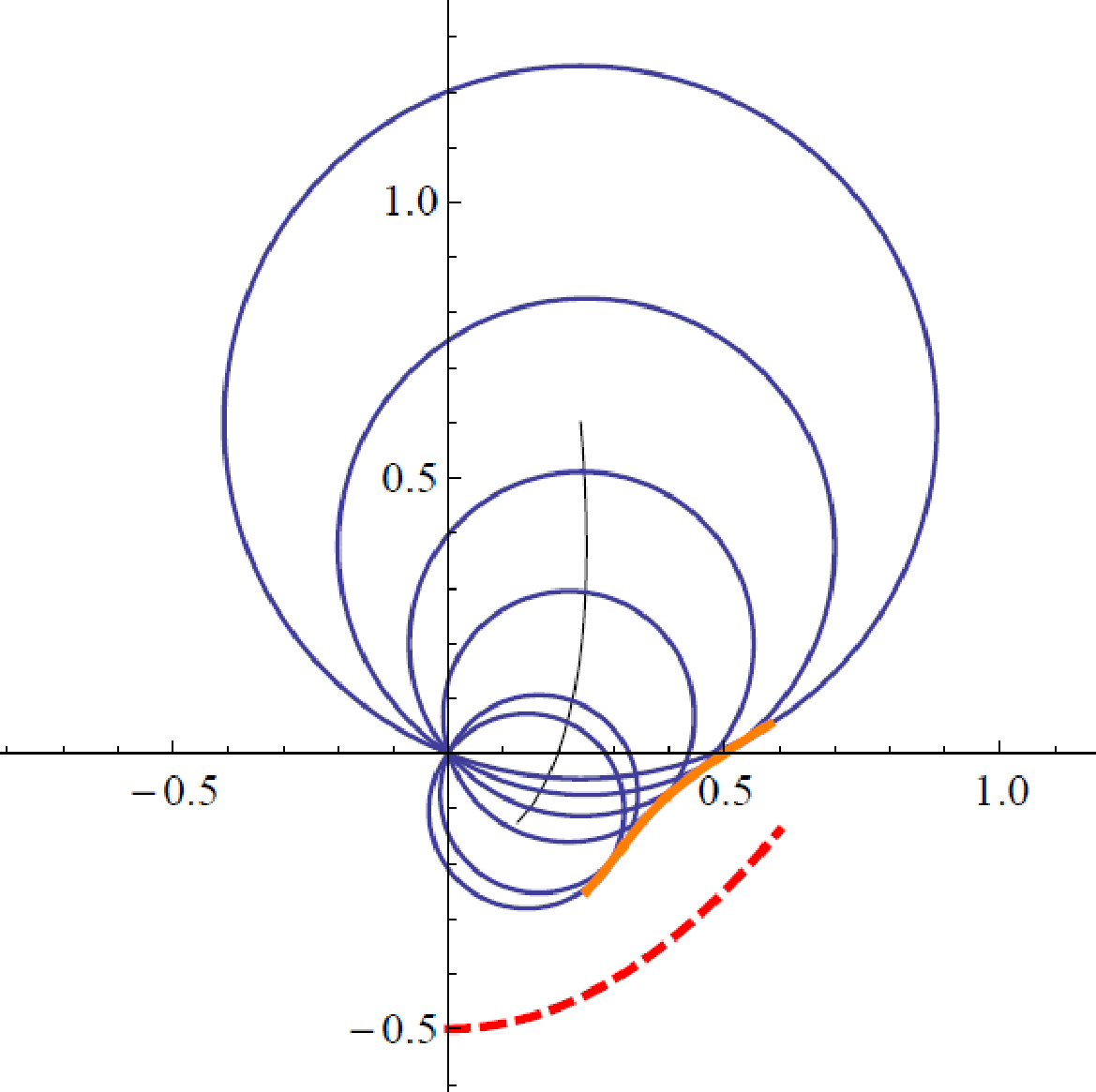} \\
				FIGURE 6. The circle family $C_{(\gamma_2,\lambda_2)}$, the loci of centers $\gamma_2(t)$ (thin),\\ the envelope $f_4(t)$ (thick) and the curve $f(t)$ (dashed).
			\end{tabular}
		\end{center}
	\end{figure}
\end{example}

\begin{example}
	Let $\gamma:\mathbb{R}_{+}\rightarrow\mathbb{R}^2$ be the mapping defined by
$\gamma(t)={\color{black}\left(
\frac{t^3}{6},\frac{t^6}{12}
\right)}$ and let $\lambda:\mathbb{R}_{+}\rightarrow\mathbb{R}_{+}$ be the function defined by $\lambda(t)=\frac{t^3\sqrt{4+t^6}}{12}$. By calculations, we obtain that
	\begin{align*}
		\mu(t)=
{\color{black}\left(-\frac{1}{\sqrt{1+t^6}},-\frac{t^3}{\sqrt{1+t^6}}\right)},~~~~\beta(t)=-\frac{t^2\sqrt{1+t^6}}{2},~~~~
		\nu(t)={\color{black}\left(-\frac{t^3}{\sqrt{1+t^6}},\frac{1}{\sqrt{1+t^6}}\right)}.
	\end{align*}
	Therefore, the function $\cos\theta(t)$ satisfying
	\begin{align*}
		\frac{d\lambda}{dt}(t)=\cos\theta(t)\beta(t)
	\end{align*}
	exists and it must have the form $\cos\theta(t)=-\frac{(2+t^6)}{\sqrt{1+t^6}\sqrt{4+t^6}}.$
	By (2) of Theorem \ref{t1.4}, we have
	\begin{align*}
		\widetilde{\nu}_1(t)=&
{\color{black}\left(-\frac{2}{\sqrt{4+t^6}},-\frac{t^3}{\sqrt{4+t^6}}\right)},~~~~~~
		f_1(t)=(0,0)\\
		\widetilde{\nu}_2(t)=&
{\color{black}\left(-\frac{2}{(1+t^6)\sqrt{4+t^6}},-\frac{3t^3+t^9}{(1+t^6)\sqrt{4+t^6}}
\right)},~~~~~~
		f_2(t)={\color{black}\left(\frac{t^9}{6(1+t^6)},-\frac{t^6}{6(1+t^6)}\right)}.
	\end{align*}
	Since the set of $t\in\mathbb{R}_{+}$ satisfying $|\frac{d\lambda}{dt}(t)| \leq |\beta(t)|$ is dense in $\mathbb{R}_{+}$, by the assertion (3) of Theorem \ref{t1.7}, we have $f_2(t)=\mathcal{P}e_{2\gamma(t),0}(t)$. Moreover, by the definition of the pedal curve of a frontal, we can verify that $f_2(t)$ is the pedal curve of $2\gamma(t)$  relative to the origin (see Figure 7).
	\begin{figure}[htbp]
		\begin{center}
			\small
			\begin{tabular}{cc}
				\includegraphics[width=70mm]{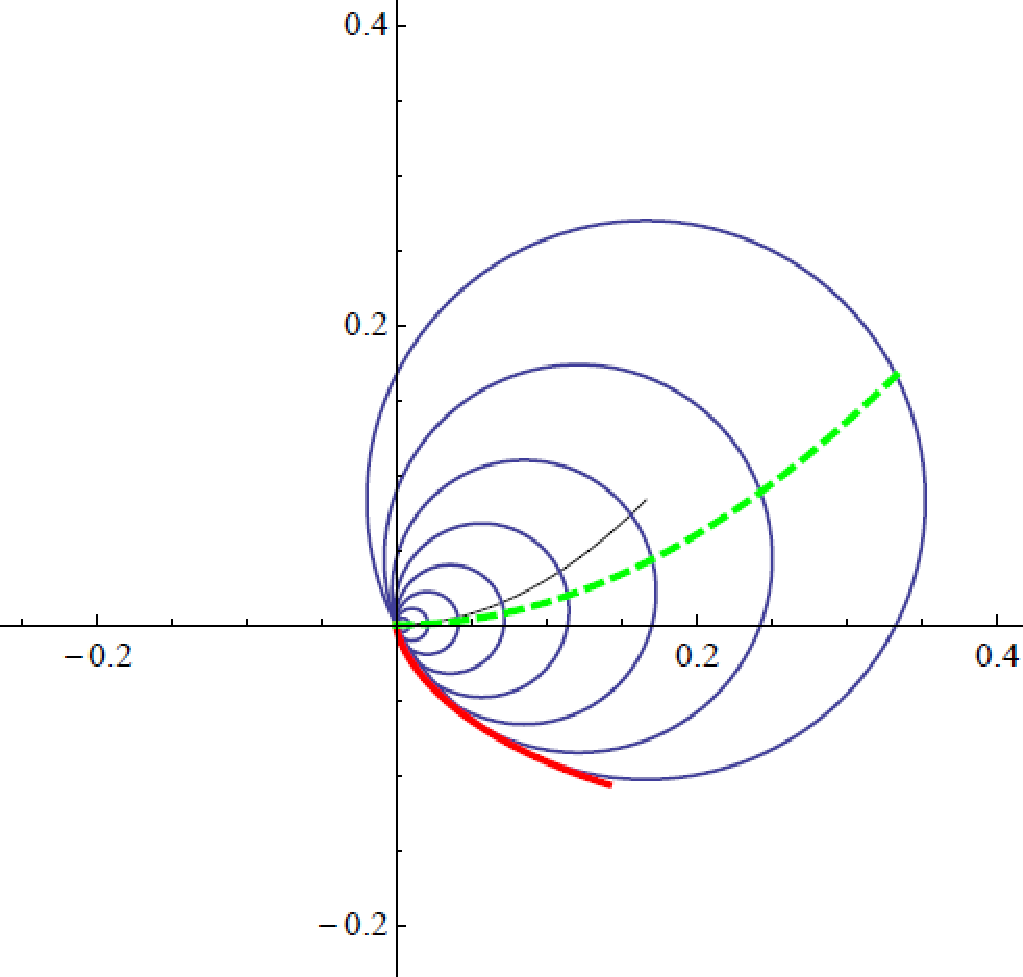} \\
				FIGURE 7. The circle family $\mathcal{C}_{(\gamma,\lambda)}$, the loci of centers $\gamma(t)$ (thin),\\ the envelope $f_2(t)$ (thick) and the curve $2\gamma(t)$ (dashed).
			\end{tabular}
		\end{center}
	\end{figure}
\end{example}

\section{Applications}\label{section4}
In this section, we investigate some applications of Theorem \ref{t1.7}.

\begin{proposition}\label{p4.1}
	Suppose that the circle family $\mathcal{C}_{(\gamma,\lambda)}$ creates a unique envelope $f(t)$, then $f:I\rightarrow\mathbb{R}^2$ is a frontal with the Gauss mapping $\widetilde{\nu}:I\rightarrow S^1$ and the curvature is
	\begin{align*}
		\bigg(l(t),~l(t)\lambda(t)\bigg).
	\end{align*}
	Moreover, $t_0$ is a singular point of $f(t)$ if and only if $t_0$ is an inflection point of $\gamma(t)$.
\end{proposition}
\begin{proof}
	By the assumption, the equality $\widetilde{\nu}(t)=\pm\mu(t)$ holds for any $t\in I$. It follows that $\theta(t)=k\pi$ is a constant from $\widetilde{\nu}(t)=-\cos\theta(t)\mu(t)\pm\sin\theta(t)\nu(t),$ where $k$ is an integer. By the assertion (1) of Theorem \ref{t1.7}, we obtain that the curvature of $f(t)$ is
	\begin{align*}
		\bigg(l(t),~l(t)\lambda(t)\bigg).
	\end{align*}
	Moreover, $t_0$ is a singular point of $f(t)$ if and only if $l(t_0)\lambda(t_0)=0$. Since $\lambda(t)\neq0$ for any $t\in I$, then $l(t_0)\lambda(t_0)=0$ is equivalent to $l(t_0)=0${\color{black},} which means $t_0$ is an inflection point of $\gamma(t)$.
\end{proof}

{\color{black}In} the case of
the circle family $\mathcal{C}_{(\gamma,\lambda)}$ creating a unique envelope
$f(t){\color{black},}$ {\color{black}b}y the assertion (2-i) of Theorem \ref{t1.7},
the centre $\gamma(t)$ is the evolute of $f(t)$.
Then {\color{black}the given circle} ${\color{black}C_{(\gamma(t), \lambda(t))}}$ is the osculating circle of {\color{black}$f$ at $t$} if $t$ is a regular point of $f$.
{\color{black}This fact can be generalized as follows}.

\begin{proposition}
	Suppose that the circle family $\mathcal{C}_{(\gamma,\lambda)}$ creates an envelope $f:I\rightarrow\mathbb{R}^2$ and $|\cos\theta(t_0)|=1$.
If $t_0$ is not an inflection point of $f$, then
{\color{black}the given circle} ${\color{black}C_{(\gamma(t_0), \lambda(t_0))}}$
is the osculating circle of {\color{black}$f$ at $t_0$}.
\end{proposition}
\begin{proof}
	By the assertion (1) of Theorem \ref{t1.7}, it follows that $f(t)$ is a frontal with the moving frame $\{\widetilde{\nu}(t), \widetilde{\mu}(t)\}$ and the curvature \begin{align*}
		l_f(t)=l(t)\pm\dot{\theta}(t),~~~~~~\beta_f(t)=\lambda(t)[l(t)\pm\dot{\theta}(t)]\pm\beta(t)\sin\theta(t).
	\end{align*}
	In other words, we have $$\frac{d}{dt}\widetilde{\nu}(t)=l_f(t)\widetilde{\mu}(t),~~~ \frac{d}{dt}\widetilde{\mu}(t)=-l_f(t)\widetilde{\nu}(t),~~~ \frac{df}{dt}(t)=\beta_f(t)\widetilde{\mu}(t).$$
	According to the proof of the assertion (1) of Theorem \ref{t1.7}, $|\cos\theta(t_0)|=1$ is equivalent to $\frac{df}{dt}(t_0)=\lambda(t_0)\frac{d}{dt}\widetilde{\nu}(t_0)$. It follows that $\lambda(t_0)l_f(t_0)=\beta_f(t_0)$ if $|\cos\theta(t_0)|=1$. In this case, $t_0$ is a regular point of $f(t)$ if and only it is not an inflection point of $f(t)$. According to the definition of the evolute of the frontal, the point $\gamma(t_0)=f(t_0)-\lambda(t_0)\widetilde{\nu}(t_0)$ is on $\mathcal{E}v(f)(t)$ and $\lambda(t_0)$ is the radii of the osculating circle of {\color{black}$f$ at $t_0$}.
\end{proof}

\begin{proposition}
	Suppose that the circle family $\mathcal{C}_{(\gamma,\lambda)}$ creates two envelopes $f_i:I\rightarrow\mathbb{R}^2$ where $i=1,2$. Then, the following three hold.
	\begin{enumerate}
		\item[(1)] {\color{black}Suppose moreover that}
$f_1(t)$ is a constant vector{\color{black}.   T}hen
$t_0$ is a singular point of $f_2(t)$ if and only if $t_0$
is an inflection point of $\gamma(t)$.
		\item[(2)] {\color{black}Suppose moreover that}
$f_1(t)$ and $f_2(t)$ are both constants and $f_1(t)\neq f_2(t)${\color{black}.
T}hen {\color{black}the set $\{\gamma(t)\; |\; t\in I\}$ belongs to}
the perpendicular bisector of $f_1f_2$.
		\item[(3)] {\color{black}Suppose moreover that}
$f_1(t)=f_2(t)=constant${\color{black}.   T}hen
{\color{black}the set $\{\gamma(t)\; |\; t\in I\}$ belongs to} a line which passes through $f_1(t)$.
	\end{enumerate}
\end{proposition}
\begin{proof}
	(1) By the assertion (1) of Theorem \ref{t1.7}, we know that the envelope $f(t)$ is a frontal with Gauss mapping $\widetilde{\nu}:I\rightarrow S^1$ and the curvature
	\begin{align*}
		l_f(t)=l(t)\pm\dot{\theta}(t),~~~~~~\beta_f(t)=\lambda(t)[l(t)\pm\dot{\theta}(t)]\pm\beta(t)\sin\theta(t).
	\end{align*}
	If $f_1(t)$ is a constant vector, then $\beta_{f_1}(t)=0$. Moreover, since $\beta_{f_1}(t)+\beta_{f_2}(t)=2\lambda(t)l(t)$, we have $\beta_{f_2}(t)=2\lambda(t)l(t)$. As $\lambda(t)>0$, $\beta_{f_2}(t_0)=0$ if and only if $l(t_0)=0$.\par
	(2) By the proof of the assertion (1), we have $\beta_{f_2}(t)=2\lambda(t)l(t)$. Since $f_2(t)$ is also a constant, it follows that $l(t)=0$ for any $t\in I$. Thus, $\gamma(t)$ is a part of a straight line.
Notice that
{\color{black}the assumption \lq\lq
both} $f_1(t)$ and $f_2(t)$ are constants{\color{black}\rq\rq\; implies} that
	\begin{align*}
		\big(\gamma(t)-f_1{\color{black}(t)}\big)
\cdot\big(\gamma(t)-f_1{\color{black}(t)}\big)=
\big(\gamma(t)-f_2{\color{black}(t)}\big)\cdot\big(\gamma(t)-f_2{\color{black}(t)}\big).
	\end{align*}
	By simplification, we have
	\begin{align*}
		2\gamma(t)\cdot(f_2{\color{black}(t)}-f_1{\color{black}(t)})-
(f_2{\color{black}(t)}\cdot f_2{\color{black}(t)}-f_1{\color{black}(t)}\cdot f_1{\color{black}(t)})=0.
	\end{align*}
	This implies that $\gamma(t)$ is a part of the line with the normal vector $f_2{\color{black}(t)}-f_1{\color{black}(t)}$ {\color{black}for any $t\in I$}.
Moreover, the line passes through the point
$\frac{f_2{\color{black}(t)}+f_1{\color{black}(t)}}{2}$ since
	\begin{align*}
		(f_2{\color{black}(t)}+f_1{\color{black}(t)})\cdot
(f_2{\color{black}(t)}-f_1{\color{black}(t)})-(f_2{\color{black}(t)}\cdot f_2{\color{black}(t)}-
f_1{\color{black}(t)}\cdot f_1{\color{black}(t)})=0.
	\end{align*}

(3) Suppose that $f_1(t)=f_2(t)=f(t)=\mbox{constant}$. Then,
by {\color{black}the} assertion (2-i) of Theorem \ref{t1.7},
it follows that $\widetilde{\nu}(t)=\pm\mu(t)$ for any $t\in I$.
	Since $\beta_{f_2}(t)=2\lambda(t)l(t)=0$ and $\lambda(t)\neq0$,
{\color{black}we have} $l(t)=0$ for any $t\in I$.
{\color{black}Thus,} {\color{black}the set $\{\gamma(t)\; |\; t\in I\}$ is a subset}
of a line. Moreover, because $f(t)=\gamma(t)+\lambda(t)\widetilde{\nu}(t)$ and
$\widetilde{\nu}(t)=\pm\mu(t)$, {\color{black}it follows that
the set $\{\gamma(t)\; |\; t\in I\}$ is a subset} of a line which passes through $f_1(t)$.
\end{proof}

In \cite{IT2},  S. Izumiya and N. Takeuchi
{\color{black}investigate}
pedaloids relative to the origin and evolutoids of plane curves. They provided a beautiful relation between evolutoids and pedaloids.
{\color{black}Finally in this paper, we give an alternative proof of their result}
as an application of 
Theorem \ref{t1.7}.

\begin{proposition}
	Let $f:I\rightarrow\mathbb{R}^2$ be a frontal with the moving frame $\{\widetilde{\nu},\widetilde{\mu}\}$ and the curvature $(l_f,\beta_f)$. Suppose that there exists a smooth function $\lambda(t)\neq0$ such that $\beta_f(t)=\lambda(t)l_f(t)$ for any $t\in I$. Then we have
	\begin{align*}
		\mathcal{P}e_{f,P}[\phi](t)=\mathcal{P}e_{\mathcal{E}v_f[\phi+\frac{\pi}{2}],P}(t),
	\end{align*}
	where $P\notin
{\color{black}\{
\mathcal{E}v_f[\phi+\frac{\pi}{2}](t)\; |\; t\in I\}}$ is
a{\color{black}n arbitrary} fixed point {\color{black}of $\R^2$}.
\end{proposition}
\begin{proof}
	Let $\gamma:I\rightarrow\mathbb{R}^2$ be a frontal 
defined by
	\begin{align*}
		\gamma(t)=f(t)-\lambda(t)\widetilde{\nu}(t).
	\end{align*}
	By the assertion (3-i) of Theorem \ref{t1.4} and the assertion (2-i) of
{\color{black}Theorem} \ref{t1.7},
{\color{black}it follows that}  the circle family $\mathcal{C}_{(\gamma,|\lambda|)}$
creates a unique envelope $f(t)$, and $\gamma(t)$ is the evolute of $f(t)$. Set $\nu(t)=\widetilde{\mu}(t)$ and $\mu(t)=-\widetilde{\nu}(t)$, then $\{\nu,\mu\}$ is a moving frame along $\gamma(t)$.
	We consider {\color{black}the} circle family $\mathcal{C}_{(\gamma_1,\lambda_1)}$ where
	\begin{align*}
		\gamma_1(t)=&\frac{\gamma(t)+\lambda(t)\sin\phi\big(-\sin\phi\mu(t)+\cos\phi\nu(t)\big)+P}{2},\\
		\lambda_1(t)=&\frac{\|\gamma(t)+\lambda(t)\sin\phi\big(-\sin\phi\mu(t)+\cos\phi\nu(t)\big)-P\|}{2}.
	\end{align*}
	{\color{black}In the case}
$\lambda(t)>0$, by the assertion (2-iv) of Theorem \ref{t1.7}, it follows
{\color{black}that}
the circle family $C_{(\gamma_1,\lambda_1)}$ creates
two envelopes $f_1(t)$ and $f_2(t)$ such that $f_1(t)=P$ and $f_2(t)=\mathcal{P}e_{f,P}[\phi](t)$ if $P\notin\mathcal{E}v_f[\phi+\frac{\pi}{2}](t)$.
Moreover, by the assertion (3) of Theorem \ref{t1.7}, we have
	\begin{align*}
		f_2(t)=\mathcal{P}e_{f,P}[\phi](t)=\mathcal{P}e_{\gamma_2,P}(t),
	\end{align*}
	where
	\begin{align*}
		\gamma_2(t)=2\gamma_1(t)-P=\gamma(t)+\lambda(t)\sin\phi\big(-\sin\phi\mu(t)+\cos\phi\nu(t)\big).
	\end{align*}
	By the proof of {\color{black}the} assertion (2-ii) {\color{black}of Theorem \ref{t1.7}},
we have
	\begin{align*}
		\gamma_2(t)=\gamma(t)+\lambda(t)\sin\phi\big(-\sin\phi\mu(t)+\cos\phi\nu(t)\big)=\mathcal{E}v_f[\phi+\frac{\pi}{2}](t).
	\end{align*}
	Thus, we obtain
	\begin{align*}
		\mathcal{P}e_{f,P}[\phi](t)=\mathcal{P}e_{\mathcal{E}v_f[\phi+\frac{\pi}{2}],P}(t)=f_2(t).
	\end{align*}
{\color{black}In the case}
$\lambda(t)<0$, {\color{black}the proof can be given similarly
as the one given in the case
$\lambda(t)>0$}.
\end{proof}


\section*{Acknowledgement}
{This work was supported by the Research Institute for Mathematical Sciences, a Joint Usage/Research Center located in Kyoto University.\par
	The first author is supported by the National Natural Science Foundation of
	China (Grant No. 12001079),
	Fundamental Research Funds for the Central Universities
	(Grant No. 3132023205) and China Scholarship Council.
	The second author (corresponding author) is 
	supported by JSPS KAKENHI (Grant No. 23K03109).
	\appendix
}
%
%

\end{document}